\documentclass[a4paper]{amsart}

\usepackage[british,english]{babel} 
\usepackage{graphics,color,pgf}
\usepackage{epsfig}
\usepackage[ansinew]{inputenc}
\usepackage[all]{xy}
\usepackage{hyperref}
\usepackage{amssymb, amsmath,amsthm, mathtools, amscd}
\usepackage{mathrsfs}
\usepackage{stmaryrd}
\usepackage{cancel} 
\usepackage{tikz}
\usepackage{tikz-cd}
\usetikzlibrary{matrix}

\newtheorem{intro_thm}{Theorem}
\newtheorem{intro_cor}[intro_thm]{Corollary}

\newtheorem{lem}{Lemma}[section]

\theoremstyle{definition}
\newtheorem{dfn}[lem]{Definition}

\theoremstyle{remark}

\newtheoremstyle{TheoremNum}
        {0.2 cm}{0.2 cm}              
        {\itshape}                      
        {}                              
        {}                     
        {.}                             
        { }                             
        {\thmname{\bfseries #1}\thmnote{ \bfseries #3}}
    \theoremstyle{TheoremNum}

\newcommand{\hh}{h} 
\newcommand{\R}{\mathbb{R}}

\DeclareMathOperator{\opp}{\mathrm{opp}}

\renewcommand{\leq}{\leqslant}
\renewcommand{\geq}{\geqslant}
\renewcommand{\setminus}{\smallsetminus}
\title[Projections onto maximal flats]{Projections from Furstenberg boundaries onto maximal flats and barycenter maps}

\author[Michelle Bucher]{
Michelle Bucher} 
\address{Universit\'e de Gen\`eve}
\email{Michelle.Bucher@unige.ch}

\author[Alessio Savini]{
Alessio Savini}
\address{University of Milano-Bicocca} 
\email{alessio.savini@unimib.it}

\thanks{The first author is supported by the Swiss National Science Foundation. The second author is supported by Indam GNSAGA.
\\
\indent Mathematics Subject Classification 2020: Primary 22E41, Secondary 57T10.
} 

\begin{document}

\begin{abstract}
Let $G$ be a semisimple connected Lie group of non-compact type with finite center. Let $K<G$ be a maximal compact subgroup and $P<G$ be a minimal parabolic subgroup. 
For any pair $(F,x)$, where $F$ is a maximal flat in $G/K$ and $x \in G/P$ is opposite to the Weyl chambers determined by $F$, we define a projection $\Phi(F,x) \in F$ which is continuous and $G$-equivariant. 

Furthermore, if $q\geq 3$, we exhibit a $G$-equivariant continuous map defined on an open subset of full measure of the space of $q$-tuples of $(G/P)^q$ with image in $G/K$. When $G$ is the orientation preserving isometries of real hyperbolic space and $q=3$, we recover the geometric barycenter of the corresponding ideal triangle.

All our proofs are constructive. 

 \end{abstract}

\maketitle

\section{Introduction}
A barycenter map on a space $X$ usually refers to a map from certain probability measures on $X$ onto $X$. A first example is Cartan's barycenter \cite{Cartan} for a finite set of points (equivalently for atomic probability measures of equal weights) in a Riemannian manifold of nonpositive curvature. Cartan's barycenter has been extended to measures with finite second moment on $\mathrm{CAT}(0)$-spaces by Austin \cite{Austin} and to measures with finite first moment on Busemann spaces by Es-Sahib and Heinich \cite{EsSahibHeinich}, Sturm \cite{Sturm} and Navas \cite{Navas} to obtain several results of ergodic nature, such as variants of Birkhoff's Theorem or of the law of large numbers. 

For nonpositively curved manifolds, it is natural to extend such barycenters to (certain) probability measures on (subsets of) the geodesic boundary of $X$. Douady and Earle \cite{Douady} initiated such constructions on the hyperbolic plane in order to extend self-maps from the circle to the Poincar\'e disk. This approach was later generalized and exploited by Besson, Courtois and Gallot \cite{BCG95,BCG96} and Francaviglia   \cite{Francaviglia} for symmetric spaces of rank one, and in the higher rank case by Connell and Farb \cite{Connell,Connell2}. The most general setting so far is for nonpositively curved manifolds with  negative Ricci curvature \cite{Connell3}. In all these articles, barycenters were one of the main tool in establishing rigidity results, most notably the Entropy Rigidity Conjecture in rank one \cite{BCG95,BCG96} or products of rank one  \cite{Connell}.

For higher rank symmetric spaces the barycentric construction from \cite{Connell2} is only defined for probability measures whose support equals the Furstenberg boundary. In contrast we produce in Corollary \ref{THE barycenter} a barycenter map which can be thought of as being defined on generic atomic probability measures of equal weights $\leq 1/3$ or equivalently on $q$-tuples of generic points, where $q\geq 3$.  In the general case, Corollary \ref{THE barycenter} will be a direct consequence of Corollary \ref{cor barycenter}  or equivalently of Theorem \ref{thm barycenter} where we exhibit a continuous and equivariant projection from   generic points in the Furstenberg boundary onto generic maximal flat. As far as we know, Theorem \ref{thm barycenter} is  new in higher rank.

We start by fixing some notation. Let $G$ be a semisimple connected Lie group of non compact type with finite center. Let $K<G$ be a maximal compact subgroup, $P<G$ be a minimal parabolic subgroup and $A<P$ a maximal split torus. Recall that the associated Weyl group is the quotient $W=T/M$, where $T=N_K(A)$ is the normalizer of $A$ in $K$ and $M=Z_K(A)$ is the centralizer of $A$ in $K$. We choose $w_0\in G$ a representative of the longest element of $W$. We denote by $\mathcal{F}_{G/K}$ the set of maximal flats in the symmetric space $G/K$ and by $F_A\in \mathcal{F}_{G/K}$ the canonical maximal flat
$$F_A:=\{ aK\mid a\in A\}\subset G/K.$$
Since there is only one $G$-orbit of maximal flats, we can write $\mathcal{F}_{G/K}=\{gF_A\mid g\in G\}$. 

For a given maximal flat $F\in  \mathcal{F}_{G/K}$, we define its \emph{boundary} $\partial F\subset G/P$ as the set of equivalence classes of Weyl chambers determined by $F$ in the Furstenberg boundary $G/P$. In the particular case of the flat $F_A$, its boundary is given by the points 
$$\partial F_A=\{ wP\mid w\in W\}.$$ 
Notice that the action of $W$ on $\partial F_A$ is well defined since two representatives of $w\in W$ differ by left multiplication by an element in $M<P$. For an arbitrary maximal flat $gF_A$ we have that $\partial(gF_A)=g(\partial F_A)$. 

A pair of points in the Furstenberg boundary $G/P$ is said to be \emph{opposite} if it is in the same $G$-orbit as the pair $(P,w_0P)$. Note that in the rank one case, opposite points are distinct pairs of points. We define an open and dense subset of $\mathcal{F}_{G/K} \times G/P$ as follows:
$$(\mathcal{F}_{G/K} \times G/P)_\mathrm{opp}:=\{(F,x)\in \mathcal{F}_{G/K} \times G/P\mid x \mathrm{\ is \ opposite \ to \ every \ point \ in \ }\partial F\}.$$

\begin{intro_thm}\label{thm barycenter}
There exists a $G$-equivariant continuous map
$$\Phi: (  \mathcal{F}_{G/K}\times G/P)_{\mathrm{opp}}\longrightarrow G/K.$$
Furthermore, $\Phi(F,x)\in F$ for every $(F,x)\in (  \mathcal{F}_{G/K}\times G/P)_{\mathrm{opp}}$.
\end{intro_thm}

In the rank one case, maximal flats are geodesics, and a pair $(\gamma,x)\in  (  \mathcal{F}_{G/K}\times G/P)_{\mathrm{opp}}$ only needs to satisfy the condition that $x$ is not one of the two endpoints of the geodesic $\gamma$. The map $\Phi$ of Theorem \ref{thm barycenter} could be taken to be the orthogonal projection of $x$ on $\gamma$, and this is indeed what we recover in our constructive proof of Theorem \ref{thm barycenter}. Note that the map $\Phi$ is not unique. We will get back to the uniqueness below. 

Since two opposite points lie on a unique maximal flat, it is natural to consider a variant of Theorem \ref{thm barycenter} defined on the following full  measure subset of triples of points in $G/P$. We define the set  $(G/P)^{(3)}$ to consist of triples $(x,y,z)\in (G/P)^3$ such that  $x$ and $y$ are opposite, and hence lie on the boundary of a unique maximal flat $F_{x,y}$, and we require further that $z$ is opposite to every point in $\partial F_{x,y}$. 

\begin{intro_cor} \label{cor barycenter} 
There exists a $G$-equivariant continuous map 
$$
\Phi: (G/P)^{(3)} \longrightarrow G/K.
$$
Furthermore, $\Phi(x,y,z)$ belongs to the unique maximal flat $F_{x,y}$ for every $(x,y,z)\in  (G/P)^{(3)} $ . 
\end{intro_cor}

\begin{proof} Take the composition of the map $(x,y,z)\mapsto (F_{x,y},z)$ with the map of Theorem \ref{thm barycenter}.\end{proof}

In fact Corollary \ref{cor barycenter} is equivalent to Theorem \ref{thm barycenter}. We will actually establish Theorem \ref{thm barycenter} as a consequence of Corollary \ref{cor barycenter} for which we will present a proof independent of Theorem \ref{thm barycenter}.

In the case when $w_0$ acts as $-1$ on the Lie algebra of $A$, which holds for the groups of type $B_n$, $C_n$, $E_7$, $E_8$, $G_2$, $D_{2n}$, and hence also for their products, we can improve Corollary \ref{cor barycenter} by replacing the domain by the bigger and more natural set $ (G/P)^{3}_{\mathrm{p-opp}} $ of triples of pairwise opposite points in $G/P$: 

\begin{intro_thm} \label{thm barycenter opp} If $w_0$ acts as $-1$ on the Lie algebra of $A$ then
there exists a $G$-equivariant continuous map 
$$
\Phi: (G/P)^{3}_{\mathrm{p-opp}} \longrightarrow G/K.
$$
Furthermore, $\Phi(x,y,z)$ belongs to the unique maximal flat $F_{x,y}$ for every $(x,y,z)\in  (G/P)^{3}_{\mathrm{p-opp}}  $ . 
\end{intro_thm}

The maps from Corollary  \ref{cor barycenter} and Theorem \ref{thm barycenter opp} lack symmetry since  the image of a triple $(x,y,z)\in (G/P)^{(3)}$ will always lie on the unique flat having $x,y$ in its boundary. This is easy to solve while also generalizing to generic $q$-tuples of points. We define the set of \emph{generic} $q$-tuples of points $(G/P)^{(q)}$ as the $q$-tuples $(x_1,\ldots,x_q)$ such that $(x_i,x_j,x_k)\in (G/P)^{(3)}$  for any distinct $1\leq i,j,k\leq q$. We also consider the set  $ (G/P)^{q}_{\mathrm{p-opp}} $ of $q$-tuples of pairwise opposite points in $G/P$.

\begin{intro_cor}\label{THE barycenter}
For every $q \geq 3$, there exists a $G$-equivariant continuous symmetric map 
$$
\mathrm{bar}_q:(G/P)^{(q)} \longrightarrow G/K.
$$
If futher $w_0$ acts as $-1$ on the Lie algebra of $A$ then
there exists  a $G$-equivariant continuous symmetric map 
$$
\mathrm{bar}_q:(G/P)^{q}_{\mathrm{p-opp}} \longrightarrow G/K.
$$

\end{intro_cor}

For $G$ the orientation preserving isometries of real hyperbolic $n$-space $x,y,z$ three distincts points in the boundary $\partial \mathbb{H}^n$ we recover as $\mathrm{bar}_3(x,y,z)\in \mathbb{H}^n$ the geometric barycenter of the ideal triangle with vertices $x,y,z$.

\begin{proof} For any finite number of points  $\xi_1,\ldots,\xi_\ell \in G/K$, Cartan's barycenter \cite{Cartan} is the unique minimizer of the convex function
$$
\xi \in G/K \longrightarrow \sum_{m=1}^\ell d(\xi,\xi_m)^2 
$$
and the assignment is $G$-equivariant and continuous. To obtain the corollary, we can apply Cartan's barycenter to the points
$$\Phi(x_i,x_j,x_k)$$
for every distincts $1\leq i,j,k\leq q$, where $\Phi$ is the map from Corollary  \ref{cor barycenter} and Theorem \ref{thm barycenter opp}, respectively.
\end{proof}

The next section presents two equivalent reformulations of Theorem   \ref{thm barycenter}, Corollary \ref{cor barycenter} and Theorem \ref{thm barycenter opp}. The first one, Theorem \ref{projection in sym sp}, only serves as a bridge towards Theorem \ref{projection N and A}, which is the version we will prove in Section \ref{section proof}. In the last section we detail the computations of the projection maps in the case of  $G=\mathrm{Isom }^+(\mathbb{H}^n)$ and $G=\mathrm{SL}(3,\mathbb{C})$. 

\section{Reformulations of   Theorem   \ref{thm barycenter}, Corollary \ref{cor barycenter} and Theorem \ref{thm barycenter opp} }

Observe that $G$ acts transitively both on maximal flats and on pairs of opposite points in $G/P$. The stabilizer of the canonical flat $F_A$ is the semidirect product $TA$, where $T=N_K(A)$ is the normalizer of $A$ in $K$ and the stabilizer of the pair $(P,w_0P)$  is the direct product $MA$, where $M=Z_K(A)$ is the centralizer of $A$ in $K$. For $x\in G/P$ we consider the set 
$$\mathrm{Opp}_x:=\{y\in G/P\mid y  \mathrm{\ is \ opposite \ to \ }x\}\subset G/P$$
and define
\begin{eqnarray*}
(G/P)_\mathrm{opp}&:=&\bigcap_{w\in W} \mathrm{Opp}_{wP},\\
(G/P)_{w_0-\mathrm{opp}}&:=&\mathrm{Opp}_P\bigcap \mathrm{Opp}_{w_0P}.
\end{eqnarray*}
It is clear that 
\begin{enumerate}
\item  $(F_A,x)\in  (  \mathcal{F}_{G/K}\times G/P)_{\mathrm{opp}}$ if and only if $x\in (G/P)_\mathrm{opp}$, 
\item  $(P,w_0P,x)\in (G/P)^{(3)}$ if and only if $x\in (G/P)_\mathrm{opp}$, 
\item $(P,w_0P,x)\in (G/P)^{3}_{\mathrm{p-opp}}$ if and only if $x\in (G/P)_{w_0-\mathrm{opp}}$. 
\end{enumerate}
Finally note that since $TA$ permutes the points in $\partial F_A$, the action of $TA$ on $G/P$ restricts to an action on $(G/P)_\mathrm{opp}$. The action of $TA$ does not in general restrict to an action on $(G/P)_{w_0-\mathrm{opp}}$, but the action of $MA$ does, as $MA$ stabilizes both $P$ and $w_0P$. As a consequence, Theorem   \ref{thm barycenter}, Corollary \ref{cor barycenter} and Theorem \ref{thm barycenter opp}  are equivalent to the first, second and third statement respectively of the following:

\begin{intro_thm}\label{projection in sym sp}
\begin{enumerate}
\item There exists a $TA$-equivariant continuous map $$\widetilde{\varphi}:(G/P)_{\mathrm{opp}} \longrightarrow F_A \subset G/K.$$
\item There exists an $MA$-equivariant continuous map $$\varphi:(G/P)_{\mathrm{opp}}\longrightarrow F_A\subset G/K.$$
\item If $w_0$ acts as $-1$ on the Lie algebra of $A$ then
there exists an $MA$-equivariant continuous map 
$$\varphi:(G/P)_{w_0\mathrm{-opp}}\longrightarrow F_A\subset G/K.$$
\end{enumerate}
\end{intro_thm}

Of course, since $M<T$ the second assertion of the theorem follows from the first one. We will however first establish the existence of an $MA$-equivariant continuous map and obtain a $TA$-equivariant one by averaging over the quotient $T/M$, which is precisely the Weyl group $W$.  

The lack of uniqueness is now evident: any $MA$-equivariant map $\varphi$ of Theorem \ref{projection in sym sp} can be changed by left multiplication with an element $a\in A$. Much worse may happen: we will see in the case of $\mathrm{SL}(3,\mathbb{C})$ two very different examples of such $MA$ and $TA$-equivariant maps. 

We finish this section with another equivalent reformulation of Theorem   \ref{thm barycenter}, Corollary \ref{cor barycenter} and Theorem \ref{thm barycenter opp}  which we will prove in the next section. To do so, first recall that the set $\mathrm{Opp}_P\subset G/P$ of points opposite to $P$ forms an open dense subset of $G/P$ parametrized by the diffeomorphism 
\begin{equation}\label{eq:N:parametrization}
\begin{array}{rcl}
\chi:N &\longrightarrow& \mathrm{Opp}_P \\ 
n& \longmapsto& nw_0P,
\end{array}\end{equation}
where $N$ is the unipotent radical of $P$ \cite[Corollary IX.1.9]{Helgason}. For $w\in W$, define
$$N_w:=\{ n\in N\mid nw_0P \mathrm{\ is \ opposite \ to \ } wP\}.$$ 
The preimage of $(G/P)_{\mathrm{opp}}$ under the diffeomorphism $\chi$ from (\ref{eq:N:parametrization}) is then by definition
$$N_\mathrm{opp}:= \bigcap_{w\in W}N_w,$$ 
whereas the preimage of $(G/P)_{w_0-\mathrm{opp}}$ is $N_{w_0}$.
The actions of $TA$ on $(G/P)_{\mathrm{opp}}$ and $MA$ on $(G/P)_{w_0-\mathrm{opp}}$ induce through $\chi$ actions on $N_\mathrm{opp}$ and $N_{w_0-\mathrm{opp}}$, respectively, which we denote by
$$\iota_h(n):= \chi^{-1}(hnw_0P),$$
for $h\in TA$ and $n\in N_{\mathrm{opp}}$, or $h\in MA$  and $n\in N_{w_0-\mathrm{opp}}$. Restricting this action to $MA$ we recover the restriction to $N_{\mathrm{opp}}$ or $N_{w_0-\mathrm{opp}}$ of the action of $MA$ on $N$ by conjugation: 

\begin{lem}\label{iota for MA} If $h\in MA$ then $\iota_h(n)=hnh^{-1}$ for every $n\in N_{w_0-\mathrm{opp}}\supset N_{\mathrm{opp}}$. 
\end{lem}
\begin{proof} We have
 $$hnw_0P=hnh^{-1} \underbrace{hw_0 P}_{=w_0P} =\underbrace{(hnh^{-1})}_{\in N}w_0P$$ 
 since $MA<P$ and $MA$ normalizes $N$. 
\end{proof}

We will also need the canonical identification of $A$ with the canonical maximal flat $F_A$ given by the $TA$-equivariant diffeomorphism
$$\begin{array}{rcl}
A&\longrightarrow &F_A\\
a&\longmapsto &aK,
\end{array}$$
where the $T$-action on $A$ is by conjugation and the $A$-action is by left multiplication, whereas the $TA$-action on $F_A$ is by left multiplication. 
We can thus reformulate Theorem \ref{projection in sym sp} equivalently as: 

\begin{intro_thm}\label{projection N and A}

\begin{enumerate}

\item There exists a $TA$-equivariant continuous map $$\widetilde{\Psi}:N_{\mathrm{opp}}\longrightarrow A,$$

\item There exists an $MA$-equivariant continuous map $$\Psi: N_{\mathrm{opp}} \longrightarrow A.$$

\item If $w_0$ acts as $-1$ on the Lie algebra of $A$ then
there exists an $MA$-equivariant continuous map  $$\Psi: N_{w_0-\mathrm{opp}} \longrightarrow A.$$
\end{enumerate}
\end{intro_thm}

%
%
\section{Proof of Theorem \ref{projection N and A}}\label{section proof}

Recall that in virtue of the Iwasawa decomposition \cite[Theorem IX.1.3]{Helgason} any element $g \in G$ can be written in a unique way as a product $ank$, where $a \in A, n \in N, k \in K$. This allows us to define the $A$-\emph{projection} (which is not a homomorphism) as 
$$\begin{array}{rccl}
\pi_A:&G& \longrightarrow &A, \\
&g=ank&\longmapsto&\pi_A(g)=a.
\end{array}
$$
Observe that 
\begin{equation}\label{M invariance projection}
\pi_A(h g)=\pi_A(g), 
\end{equation}
for any $h \in M$. Indeed
$$h g=hank=h a h^{-1} h n h^{-1} h k=a\underbrace{h n h^{-1}}_{\in N} \underbrace{h k }_{\in K},$$
where we exploited the fact that $h\in M<K$ centralizes $A$ and normalizes $N$. 

\begin{dfn}\label{def psi}
Given an element $w \in W$, we define the $w$-\emph{projection} on $A$ as
$$\begin{array}{rccl}
\psi_w:&N_{\opp}& \longrightarrow &A\\
 &n&\longmapsto & \psi_w(n):=\pi_A(w_0^{-1}n^{-1}w^{-1} \iota_w(n)w_0),
\end{array}$$
where we are abusively considering representatives $w, w_0 \in T$. 
\end{dfn}

\begin{lem}\label{lem MA psiw}
For any element $w \in W$, the function $\psi_w$ is continuous and does not depend on the choice of the representatives of $w$ nor $w_0$. Additionally we have that
\begin{equation}\label{eq psiw M invariance}
\psi_w(h n h^{-1})=\psi_w(n),
\end{equation}
\begin{equation}\label{eq psiw A invariance}
\psi_w(a n a^{-1})=(w_0^{-1} (a (wa^{-1} w^{-1}) )w_0)\psi_w(n), 
\end{equation}
for any $a \in A, h \in M, n\in N_{\opp}$. 
\end{lem}

\begin{proof} The continuity is clear from the continuity of $\pi_A$ and $\iota_w$. The fact that $\psi_w$ does not depend on the choice of the representative of $w_0$ follows from (\ref{M invariance projection}) and the definition of $\pi_A$. For the proof of the independence of $w$, suppose that  $hw$, where $ h \in M$, is another representative for $w$. (Although $W=T/M$ is the quotient of $T$ by $M$ on the right, the left coset $wM$ is equal to the right coset $Mw$ since $M$ is normal in $T$.) Replacing $w$ by $\hh w$ (in two places) in the definition of the $w$-projection we obtain
\begin{align*}
\psi_{\hh w}(n)&=\pi_A(w_0^{-1}n^{-1}(\hh w)^{-1} \underbrace{\iota_{\hh w}(n)}_{=\iota_\hh(\iota_w(n))} w_0)\\
&=\pi_A(w_0^{-1}n^{-1}w^{-1}  \hh^{-1} h \iota_{w}(n)h^{-1} w_0)\\
&=\pi_A(w_0^{-1}n^{-1}w^{-1}\iota_w(n) \underbrace{h^{-1} w_0}_{\in K}),
\end{align*}
where we used the fact that $\iota$ is an action and that it restricts to conjugation on $MA$ (Lemma \ref{iota for MA}). The last evaluated expression differs from $w_0^{-1}n^{-1}w^{-1}\iota_{w}(n) w_0$ by an element in $K$ on the right, which has no effect on the $\pi_A$ projection and shows that $\psi_{w\hh}=\psi_w$. 

Let now $\hh\in M$. We have
 \begin{align*}
\psi_w(\hh n\hh^{-1})&=\pi_A(w_0^{-1}(\hh n \hh^{-1})^{-1} w^{-1} \iota_w(\hh n\hh^{-1})w_0) && \mathrm{by \ definition,} \\
&=\pi_A(w_0^{-1}
\hh n^{-1}  \hh^{-1}w^{-1} \iota_{w\hh}(n)w_0) &&  \mathrm{by\ Lemma \ }\ref{iota for MA},\\
&=\pi_A(\underbrace{(w_0^{-1} \hh w_0)}_{\in M} w_0^{-1} n^{-1} (wh)^{-1}\iota_{wh}(n)w_0) &&\\
&=\pi_A( w_0^{-1} n^{-1} (wh)^{-1}\iota_{wh}(n)w_0), &&\mathrm{by \ }(\ref{M invariance projection})\\
&=\psi_{wh}(n)=\psi_w(n),&&\\
\end{align*}
where for the last equality we used the independence of $\psi_w$ on the choice of representative $w\in T$. 

For the conjugation by an element $a\in A$ we preliminary compute
\begin{align}
\label{iota ana} \iota_w(ana^{-1})&=\iota_w(\iota_a(n))&&  \mathrm{by\ Lemma \ }\ref{iota for MA},\\
\nonumber &=\iota_{wa}(n)&& \mathrm{since \ }\iota \mathrm{ \ is \ an \ action},\\
\nonumber&=\iota_{waw^{-1}w}(n)&&\\
\nonumber&=\iota_{waw^{-1}}\iota_w(n)&& \mathrm{since \ }\iota \mathrm{\ is \ an \ action},\\
\nonumber &=waw^{-1} \iota_w(n) (waw^{-1})^{-1} &&  \mathrm{by\ Lemma \ }\ref{iota for MA} \mathrm{ \ for \ }waw^{-1}\in A.
\end{align}
Using this relation we obtain 
\begin{align*}
\psi_w(a n a^{-1})&=\pi_A(w_0^{-1}(a n a^{-1})^{-1} w^{-1} \iota_w(a n a^{-1})w_0) && \mathrm{by \ definition,} \\
&=\pi_A(w_0^{-1}(a n a^{-1})^{-1} w^{-1} (w a w^{-1})\iota_w(n)(w a w^{-1})^{-1} w_0) && \\
&=\pi_A(\underbrace{(w_0^{-1} a w_0)}_{\in A} \underbrace{w_0^{-1} n^{-1} w^{-1}\iota_w(n)w_0}_{\in P}\underbrace{(w_0^{-1} (wa w^{-1})^{-1} w_0)}_{\in A}).&
\end{align*}
Now we crucially need the fact that, although the projection $\pi_A$ is not a homomorphism on $G$, it is one when restricted to $P$ (corresponding to taking the quotient by the normal subgroup $NM\lhd P$). Observe that the middle expression indeed belongs to $P$ since by definition, 
$$wnw_0P=\iota_w(n)w_0P.$$
The lemma now follows given that the evaluation of $\pi_A$ on this middle expression is precisely $\psi_w(n)$ and $\pi_A$ restricted to $A$ is the identity. 
\end{proof}

\begin{proof}[Proof of Theorem \ref{projection N and A} (3)]  Observe that the definition \ref{def psi} of $\psi_{w_0}$ extends to $N_{w_0-\mathrm{opp}}$ and that the relations established in Lemma \ref{lem MA psiw} still hold for $n\in N_{w_0-\mathrm{opp}}$ when $w=w_0$. 

Suppose now that $w_0$ acts as $-1$ on the Lie algebra of $A$, or equivalently $w_0aw_0^{-1}=a^{-1} $ for any $a\in A$. In this case Equation (\ref{eq psiw A invariance}) for $w=w_0$ becomes 
\begin{equation}\label{proof 5 3}
\psi_{w_0}(a n a^{-1})=a^{-2}\psi_{w_0}(n).\end{equation}
We define 
$$\Psi(n):=\psi_{w_0}(n)^{-\frac{1}{2}},$$
where $a^\lambda=\mathrm{exp}(\lambda \log a)$ for any $a\in A, \lambda\in \R$. This function is continuous since $\psi_{w_0}$ is. It is $M$-invariant by (\ref{eq psiw M invariance}) and $A$-equivariant by (\ref{proof 5 3}).\end{proof} 

In the general case we will exploit the relation $\prod_{w\in W} waw^{-1}=e$ for every $a\in A$ and define $\Psi$ as an average of all the $\psi_w$'s.

\begin{proof}[Proof of Theorem \ref{projection N and A} (2)] Define
$$\begin{array}{rccl}
\Psi:&N_{\opp} &\longrightarrow &A\\
&n&\longmapsto & \Psi(n):=\left( \prod_{w \in W} w_0 \psi_w(n)w_0^{-1}\right)^{\frac{1}{|W|}} .
\end{array}$$
 Note that $\Psi$ is independent of all the choices of the representatives $w\in W$ since this is the case for the $\psi_w$'s by Lemma \ref{lem MA psiw} and the conjugation $w_0 a w_0^{-1}$ is independent of the choice of representative $w_0$ for $a\in A$ (since $M$ commutes with $A$). 

The continuity of  $\Psi$ is immediate from the continuity of the $\psi_w$'s. Since $M$ acts trivially on $A$, the $M$-equivariance amounts to the $M$-invariance which is an immediate consequence of Equation \eqref{eq psiw M invariance}. We are left to show the $A$-equivariance of $\Psi$. We have that
\begin{align*}
\Psi(ana^{-1})&=\left( \prod_{w \in W} w_0 \psi_w(ana^{-1})w_0^{-1}\right)^{\frac{1}{|W|}} \\
&= \left( \prod_{w \in W} w_0 \left( (w_0^{-1} (a (w a^{-1} w^{-1}) )w_0)\psi_w(n) \right))w_0^{-1}\right)^{\frac{1}{|W|}} \\
&= \left( \prod_{w \in W} a (w^{-1}a^{-1} w)(w_0\psi_w(n) w_0^{-1})\right)^{-\frac{1}{|W|}} \\
&=\left(     a^{|W|} \prod_{w \in W} (w^{-1}a^{-1} w) \prod_{w \in W}(w_0\psi_w(n) w_0^{-1})     \right)^{\frac{1}{|W|}}\\
&=a \Psi(n),
\end{align*}
where we simply used the definition of $\Psi$, the relations for $\psi_w(ana^{-1})$ from Equation (\ref{eq psiw A invariance}) and the fact that $\prod_{w \in W} (waw^{-1})=e$. \end{proof}

If $G$ is a product of rank $1$ Lie groups then the projections proposed in the proof of Theorem \ref{projection N and A} (2) and (3) coincide. In that case also $N_\mathrm{opp}=N_{w_0-\mathrm{opp}}$. In general however the inclusion $N_\mathrm{opp}\subset N_{w_0-\mathrm{opp}}$ is strict and it seems unlikely that the two projections could always agree. 

\begin{proof}[Proof of Theorem \ref{projection N and A} (1)] 
Now that the second item of Theorem \ref{projection N and A} is established we can exploit the existence of an $MA$-invariant map $\Psi:N_{\mathrm{opp}}\rightarrow A$ to produce a $TA$-equivariant one. We define
\begin{equation}\label{projection A W equivariant}
\begin{array}{rccl}
\widetilde{\Psi}:&N_{\mathrm{opp}}& \longrightarrow &A\\
 &n&\longmapsto&\widetilde{\Psi}(n)=\left(\prod_{w \in W} w^{-1}\Psi(\iota_w(n))w\right)^{\frac{1}{|W|}},
\end{array}
\end{equation}
where by $w\in W$ we mean a choice of representative in $T$. Let us right away verify that the product is independent of the choice of representative $w$ in $T$. Let $hw$, for $h\in M$, be another representative. We have
$$\Psi(\iota_{hw}(n))=\Psi(\iota_h(\iota_w(n)))=\Psi(h\iota_w(n) h^{-1})=\Psi(\iota_w(n)),$$
where we used the fact that $\iota$ is an action which restricts to the action by conjugation for $h\in M$ and the $M$-invariance of $\Psi$. Furthermore, since $A$ and $M$ commute, conjugation of $\Psi(\iota_w(n))\in A$ by $w^{-1}$ or $(hw)^{-1}$ gives the same result.

First we show that $\widetilde{\Psi}$ is still $A$-equivariant: For $a \in A$ we have 
\begin{align*}
\widetilde{\Psi}(ana^{-1})&=\left(\prod_{w \in W} w^{-1} \Psi(\iota_w(ana^{-1})) w\right)^{\frac{1}{|W|}}\\
&=\left(\prod_{w \in W} w^{-1} \Psi((waw^{-1})\iota_w(n)(waw^{-1})^{-1}) w\right)^{\frac{1}{|W|}}\\
&=\left(\prod_{w \in W} w^{-1}(waw^{-1})\Psi(\iota_w(n)) w\right)^{\frac{1}{|W|}}=\\
&=\left(a^{|W|} \prod_{w \in W} w^{-1}\Psi(\iota_w(n))w\right)^{\frac{1}{|W|}}=a \widetilde{\Psi}(n),
\end{align*}
where we used the Expression (\ref{iota ana}) for $\iota_w(ana^{-1})$ and the $A$-equivariance of $\Psi$. 

Now we prove that  $\widetilde{\Psi}$ has become $T$-equivariant. For $u \in T$ we have
\begin{align*}
\widetilde{\Psi}(\iota_{u}(n))&=\left(\prod_{w \in W} w^{-1} \Psi(\iota_w(\iota_{u}(n))) w\right)^{\frac{1}{|W|}}\\
&=\left(\prod_{w \in W} w^{-1} \Psi(\iota_{wu}(n)) w\right)^{\frac{1}{|W|}}\\
&=\left(\prod_{v \in W} (vu^{-1})^{-1} \Psi(\iota_v(n)) vu^{-1}\right)^{\frac{1}{|W|}}\\
&=\left(u\widetilde{\Psi}(n)^{|W|}u^{-1}\right)^{\frac{1}{|W|}}=u\widetilde{\Psi}(n)u^{-1},
\end{align*}
where we made the change of variable $v=wu$.\end{proof}

\section{Examples of projections} 

\subsection*{The case of $G=\mathrm{Isom}^+(\mathbb{H}^n)$}
We realize the boundary at infinity of $\mathbb{H}^n$ as $\mathbb{R}^{n-1} \cup \{ \infty \}$ and consider the minimal parabolic subgroup $P=\mathrm{Stab}_G(\infty)$ which identifies with the group of orientation preserving similarities of $\mathbb{R}^{n-1}$. We take as  maximal split torus $A$ the subgroup of homotheties $\{ a_\lambda: x \mapsto \lambda x \ | \ \lambda > 0\}$. The unipotent radical $N$ of $P$ then corresponds to the subgroup of translations, namely 
$$
\begin{array}{rcl}
\mathbb{R}^{n-1} &\longrightarrow &N \\
 v &\longmapsto & n_v(x):=x+v.
 \end{array}
$$
Observe that this identification is precisely the inverse of the diffeomorphism $\chi$ considered in (\ref{eq:N:parametrization}) between $N$ and opposite points to $\infty$, which are in rank $1$ distinct points from $\infty$. 

The boundary of the canonical flat consists of the two points $\infty$ and $0$. The set $(G/P)_{\mathrm{opp}}$ corresponds to the points in the boundary  distincts from $\infty$ and $0$. As a consequence, we can identify $N_{\mathrm{opp}}$ with $\mathbb{R}^{n-1} \setminus \{0\}$. The action on $N_\mathrm{opp}$ of any representative $w_0$ of the longest element in the Weyl group is thus simply given by 
$$
\begin{array}{rccl}
\iota_{w_0}:&N_\mathrm{opp}&\longrightarrow &N_\mathrm{opp}\\
&n_v&\longmapsto & n_{w_0(v)}.
\end{array}
$$

We fix as representative $w_0$  the inversion with respect to the unit sphere in the upper half space model pre-composed with a reflection fixing $0$ and $\infty$ in order for the composition to preserve orientation. On the boundary (and also on the upper half space) we obtain the expression 
$$\begin{array}{rccl}
w_0:&\mathbb{R}^{n-1}\cup \{\infty\} & \longrightarrow &\mathbb{R}^{n-1}\cup \{\infty\}\\ 
&x&\longmapsto &w_0(x):=\frac{x-2\langle e_1,x\rangle e_1}{\lVert x \rVert^2},
\end{array}$$
where $e_1=(1,0,\ldots,0)$ is the first vector of the canonical basis of $\mathbb{R}^n$, $\langle \cdot \ , \ \cdot\rangle$ denotes the standard scalar product and $\lVert \ \cdot \ \rVert$ the associated norm.

In the proof of Theorem \ref{projection N and A} (3)we define the projection $\Psi$ as
$$
\Psi(n_v)=\psi_{w_0}(n_v)^{-\frac{1}{2}}, 
$$
where by  Definition \ref{def psi} we have that
$$
\psi_{w_0}(n_v)=\pi_A(w_0n_{-v}w_0\iota_{w_0}(n_v)w_0),
$$
where we used that $w_0=w_0^{-1}$ and $(n_v)^{-1}=n_{-v}$. By \cite[Lemma 12]{BucSavExp} the projection $\pi_A$ is $a_\lambda$, where $\lambda$ is the last coordinate of the image of $e_n=(0,\ldots,0,1)\in \mathbb{H}^n$. Since translations and our element $w_0$ admit the same expression on the upper half space and on the boundary, it just remains to compute
\begin{align*}
(w_0n_{-v}w_0\iota_{w_0}(n_v)w_0)(e_n)&=(w_0n_{-v}w_0n_{w_0(v)})(e_n)\\
&=(w_0 n_{-v} w_0)(e_n+w_0(v))\\
&=(w_0 n_{-v})\left(\frac{1}{1+\lVert v \rVert^2}(v+\lVert v \rVert^2e_n)\right)\\
&=w_0\left( \frac{\lVert v \rVert^2}{1+\lVert v \rVert^2}(e_n-v)\right)=\frac{(w_0(-v)+e_n)}{\lVert v \rVert^2}.
\end{align*}
It follows that 
$$\psi_{w_0}(n_v)=a_{(1/\|v\|^2)}$$
and
$$
\Psi(n_v)=a_{\lVert v \rVert}.
$$
Further using the canonical identification of $A$ with the canonical flat, $a_{\lVert v \rVert}$ corresponds to ${\lVert v \rVert}\cdot e_n$ which is precisely the orthogonal projection of $v$ on the geodesic determined by $0$ and $\infty$, as claimed in the introduction. In this case we have that the map $\Psi$ is also equivariant with respect to the action of the Weyl group $W=\{e,w_0\}$, thus we have that $\Psi=\widetilde{\Psi}$. 

\subsection*{The case of $G=\mathrm{SL}(3,\mathbb{C})$} In this case we take $P$ to be the subgroup of upper triangular matrices.  The maximal split torus $A$ is then the subgroup of diagonal matrices with positive real diagonal entries,
$$
A=\{ \mathrm{diag}(a_{11},a_{22},a_{33}) \ | \ a_{11}a_{22}a_{33}=1, \ a_{ii} >0 \textup{\ \ for $i=1,2,3$} \}. 
$$
The unipotent radical $N$ of $P$ is given by the subgroup of unipotent matrices and $M$ consists of the determinant $1$ diagonal matrices with diagonal entries in $\{\pm 1\}$. The Weyl group is isomorphic to the symmetric group $\mathrm{Sym}(3)$. Representatives $s,t\in \mathrm{SL}(3,\mathbb{C})$ of the generators of $W$ can be taken as
$$
s=
\left(
\begin{array}{rrr}
0 & -1 & 0 \\
1 & 0 & 0  \\
0 & 0 & 1 
\end{array}
\right) \quad \mathrm{and} \quad t=
\left(
\begin{array}{rrr}
1 & 0 & 0  \\
0 & 0 & -1\\
0 & 1 & 0  
\end{array}
\right).$$
Representatives of the remaining nontrivial elements in $W$ are
$$
st=
\left(
\begin{array}{rrr}
0 & 0 & 1\\
1 & 0 & 0  \\
 0 & 1 & 0 
\end{array}
\right), \quad  ts=
\left(
\begin{array}{rrr}
 0 & -1& 0 \\
 0 & 0 & -1\\
1 & 0 & 0  \\
\end{array}
\right) \quad \mathrm{and} \quad sts=tst=w_0=
\left(
\begin{array}{rrr}
0 & 0 & 1 \\
0 & -1 & 0  \\
1 & 0 & 0 
\end{array}
\right).$$

To determine $\iota_w$ for any of these representatives $w$ of elements in $ W$ we need, for any unipotent matrix $n\in N$,  to find a unipotent matrix $\iota_w(n)$ such that $wnw_0P=\iota_w(n)w_0P$ or equivalently $w_0^{-1}n^{-1}w^{-1}\iota_w(n)w_0\in P$, which can be restated as $n^{-1}w^{-1}\iota_w(n)\in w_0Pw_0^{-1}$. The latter group is the group of lower triangular matrices. We will detail the computations in the case of $w=w_0$. Let 
$$n=\left(
\begin{array}{ccc}
1 & x & z \\
0 & 1 & y \\
0 & 0 & 1 \\
\end{array}
\right)$$
and suppose that $i_{w_0}(n)$ is given as
$$i_{w_0}(n)=\left(
\begin{array}{ccc}
1 & u & w \\
0 & 1 & v \\
0 & 0 & 1 \\
\end{array}
\right),$$
for some $u,v,w\in \mathbb{C}$, whose dependency in $x,y,z$ we need to establish. We need the product $n^{-1}w^{-1}\iota_w(n)$ to be lower diagonal, so we compute the upper triangular entries of the product
$$
\left(
\begin{array}{ccc}
1 & -x & -z+xy \\
0 & 1 & -y \\
0 & 0 & 1 \\
\end{array}
\right)\left(
\begin{array}{ccc}
0 & 0 & 1 \\
0 & -1 & 0  \\
1 & 0 & 0 
\end{array}
\right)\left(
\begin{array}{ccc}
1 & u & w \\
0 & 1 & v \\
0 & 0 & 1 \\
\end{array}
\right)$$
\begin{equation}\label{matrix}=\left(
\begin{array}{ccc}
* & x+u(-z+xy) & 1+xv+w(-z+xy) \\
* & * &-v-yw  \\
* & * & * \\
\end{array}
\right).
\end{equation}
These three computed entries have to be zero, which is equivalent to 
\begin{equation}\label{values}u=-\frac{x}{xy-z}, \quad v=-\frac{y}{z}, \quad w=\frac{1}{z}\end{equation}
and thus 
$$
i_{w_0}
\left(
\begin{array}{ccc}
1 & x & z \\
0 & 1 & y \\
0 & 0 & 1 \\
\end{array}
\right)=
\left(
\begin{array}{ccc}
1 & -\frac{x}{xy-z} & \frac{1}{z}\\
0 & 1 & -\frac{y}{z} \\
0 & 0 & 1\\ 
\end{array}
\right).
$$
Observe that we deduce that $n \in N_{w_0}$  if and only if $z\neq 0$ and $z-xy \neq 0$. It remains to compute $\Psi_{w_0}(n)$, which is by definition the $A$-projection of $w_0^{-1}n^{-1}w^{-1}\iota_w(n)w_0$. The latter is the conjugation by $w_0^{-1}$ of the matrix  (\ref{matrix}) where we plug in $u,v,w$ the values of (\ref{values}). Up to sign, this conjugation exchanges the first and last diagonal entries, and the $A$-projection of an upper diagonal matrix is precisely given by the absolute value of its diagonal entries (see \cite[Lemma 25]{BucSavExp}). As a consequence, $\Psi_{w_0}(n)$ is simply given by the absolute value of the diagonal entries of the matrix in  (\ref{matrix}) with the values of  (\ref{values}) taken in reverse order. Since we are actually interested in $w_0\Psi_{w_0}(n)w_0^{-1}$ we forget about taking the entries in reverse order. We thus obtain
$$
w_0\psi_{w_0}
\left(
\begin{array}{ccc}
1 & x & z\\
0 & 1 & y\\
0 & 0 & 1\\
\end{array}
\right)w_0^{-1}=
\left(
\begin{array}{ccc}
|-z+xy|& 0 & 0\\
0 & \frac{|z|}{|-z+xy|} &0\\
0 & 0 & \frac{1}{|z|}
\end{array}
\right).
$$

The computation of the remaining $\iota_w$ and $\psi_w$ are completely analogous. We obtain
\begin{align*}\iota_s(n)&=\left(
\begin{array}{ccc}
1 & -\frac{1}{x} & -y\\
0 & 1 & z \\
0 & 0 & 1\\ 
\end{array}
\right), \quad&w_0\psi_s(n)w_0^{-1} =\left(
\begin{array}{ccc}
|x|&0 &0\\
0 &\frac{1}{|x|} &0\\
0 & 0 & 1\\ 
\end{array}
\right),\\
 \iota_t(n)&=\left(
\begin{array}{ccc}
1 &{-z+xy} & \frac{z}{y}\\
0 & 1 & -\frac{1}{y} \\
0 & 0 & 1\\ 
\end{array}
\right), \quad &w_0\psi_t(n)w_0^{-1}=\left(
\begin{array}{ccc}
1 &0 &0\\
0 & |y| &0 \\
0 & 0 & \frac{1}{|y|}\\ 
\end{array}
\right),\\
\iota_{ts}(n)&=\left(
\begin{array}{ccc}
1 & \frac{-z+xy}{x} & -\frac{y}{z}\\
0 & 1 & -\frac{1}{z} \\
0 & 0 & 1\\ 
\end{array}
\right), \quad&w_0\psi_{st}(n)w_0^{-1} =\left(
\begin{array}{ccc}
|x| & 0&0\\
0 & \frac{|z|}{|x|} & 0 \\
0 & 0 & \frac{1}{|z|}\\ 
\end{array}
\right),\\
 \iota_{st}(n)&=\left(
\begin{array}{ccc}
1 & \frac{1}{z-xy} & \frac{1}{y}\\
0 & 1 & \frac{z}{y} \\
0 & 0 & 1\\ 
\end{array}
\right), \quad &w_0\psi_{st}(n)w_0^{-1}=\left(
\begin{array}{ccc}
|-z+xy|&0 &0\\
0 & \frac{|y|}{|-z+xy|} & 0 \\
0 & 0 & \frac{1}{|y|}\\ 
\end{array}
\right).
\end{align*}
It remains to take the $6$-th root of the product of all the $w_0\psi_ww_0^{-1}$'s to obtain
\begin{equation}\label{projection sl3}
\Psi(n)=\left(
\begin{array}{ccc}
|x|^2|-z+xy|^2&0 &0\\
0 & \frac{|y|^2|z|^2}{|x|^2|-z+xy|^2} & 0 \\
0 & 0 & \frac{1}{|y|^2|z|^2}\\ 
\end{array}
\right)^{1/6}.
\end{equation}
Now recall that $AM$ consists of diagonal matrices of determinant $1$ with real entries, and the action by conjugation of a diagonal matrix with diagonal entries $a_1,a_2,a_3$ on $N$ sends $n$ to
\begin{equation}\label{action}\left(
\begin{array}{ccc}
1 & \frac{a_1}{a_2}x &  \frac{a_1}{a_3}z\\
0 & 1 &  \frac{a_2}{a_3}y\\
0 & 0 & 1\\
\end{array}
\right).\end{equation}
It is straightforward to check that $\Psi$ is indeed $AM$-equivariant. 

By staring at the action given by (\ref{action}), another $AM$-equivariant map jumps to the eyes: simply take
\begin{equation}\label{different projection sl3}
\Psi'(n)=\left(
\begin{array}{ccc}
|x|^2|z|^2&0 &0\\
0 & \frac{|y|^2}{|x|^2} & 0 \\
0 & 0 & \frac{1}{|y|^2|z|^2}\\ 
\end{array}
\right)^{1/6}.
\end{equation}
The simpler form of the latter projection can seem more appealing, but for applications to constructing continuous cocycles on the Furstenberg boundary we will, in our upcoming joint paper, really exploit the algebraic properties of the former projection $\Psi$. 

Neither $\Psi$ nor $\Psi'$ are equivariant for the action by $W$. Using the averaging procedure given by  (\ref{projection A W equivariant}) in the proof of Theorem \ref{projection N and A} (1) it is immediate to check that we obtain from $\Psi$ and $\Psi'$ the two $TA$-equivariant maps 

$$
\widetilde{\Psi}(n)=
\left(
\begin{array}{ccc}
\frac{|xy-z||x||z|^2}{|y|} & 0 & 0\\
0 & \frac{|xy-z||y|^2}{|x|^2|z|} & 0\\
0 & 0 & \frac{|x|}{|xy-z|^2|y||z|}
\end{array}
\right)^{\frac{1}{6}}.
$$
and
$$
\widetilde{\Psi}'(n)=
\left(
\begin{array}{ccc}
\frac{|xy-z|^{\frac{2}{3}}|x|^{\frac{2}{3}}|z|^{\frac{8}{3}}}{|y|^{\frac{4}{3}}} & 0 & 0\\
0 & \frac{|xy-z|^{\frac{2}{3}}|y|^{\frac{8}{3}}}{|x|^{\frac{4}{3}}|z|^{\frac{4}{3}}} & 0\\
0 & 0 & \frac{|x|^{\frac{2}{3}}}{|xy-z|^{\frac{4}{3}}|y|^{\frac{4}{3}}|z|^{\frac{4}{3}}}
\end{array}
\right)^{\frac{1}{6}}
$$
respectively.

\bibliographystyle{alpha}

\begin{thebibliography}{20} 
%
\bibitem{Austin} Austin T., \emph{A $\mathrm{CAT}(0)$-valued pointwise ergodic theorem.}
J. Topol. Anal.   3 (2011), no. 2, 145--152.
%
\bibitem{BCG95} Besson G., Courtois G., Gallot S., \emph{Entropies et rigidit\'es des espaces localement sym\'etriques de couboure n\'egative}, Geom. Func. Anal. {\bf 5} (1995), n. 5, 731--799.

\bibitem{BCG96} Besson G. Courtois G. Gallot S., \emph{Minimal entropy and Mostow's rigidity theorems}, Ergodic Theory Dynam. Systems {\bf 16} (1996), 623--649.

\bibitem{Bishop} Bishop R. L., O' Neil B., \emph{Manifolds of negative curvature}, Trans. Amer. Math. Soc. {\bf 145} (1969), 1--49.
%
%
%
\bibitem{BucSavExp} Bucher M, Savini A., \emph{Some explicit cocycles on the Furstenberg boundary for products of isometries of hyperbolic spaces and $\mathrm{SL}(3,\mathbb{K})$}, preprint.
%
%
%
%
%
%
%
%
%

\bibitem{Cartan} E. Cartan, \emph{Le\c cons sur la geometrie des espaces de Riemann,} Second augmented edition,
Gauthier-Villars, Paris 1963 (1st ed. 1928)

\bibitem{Connell} Connell C., Farb B., \emph{Minimal entropy rigidity for lattices in products of rank one symmetric spaces}, \emph{Comm. Anal. Geom.} {\bf 11} (2003), n. 5, 1001--1026.

\bibitem{Connell2} Connell C., Farb B., \emph{The degree theorem in higher rank}, J. Diff Geom. {\bf 65} (2003), n. 1, 19--59.

\bibitem{Connell3} Connell C., Farb B., \emph{Some recent applications of the barycenter method in geometry}. Topology and geometry of manifolds (Athens, GA, 2001), 19?50.
Proc. Sympos. Pure Math., 71
American Mathematical Society, Providence, RI, 2003.
%
\bibitem{Douady} Douady E., Earle C., \emph{Conformally natural extension of homeomorphisms of the circle}, Acta Math. {\bf 157} (1986), 23--48.

\bibitem{EsSahibHeinich}  Es-Sahib A., Heinich H., \emph{Barycentre canonique pour un espace m\'etrique \`a courbure n\'egative}.
S\'eminaire de Probabilit\'es XXXIII. Lecture Notes in Math. 1709, Springer, Berlin (1999), 355-3.
%
%
\bibitem{Francaviglia} Francaviglia S., \emph{Constructing equivariant maps for representations}, Ann. Inst. Fourier {\bf 59} (2009), n. 1, 393--428.

%
%

\bibitem{Helgason} Helgason S., \emph{Differential geometry, Lie groups, and symmetric spaces}, Corrected reprint
of the 1978 original, Grad. Stud. Math. 34, Amer. Math. Soc., Providence, R.I., 2001.

\bibitem{Navas} Navas A., \emph{ An  $L^1$  ergodic theorem with values in a non-positively curved space via a canonical barycenter map
}, Ergodic Theory Dynam. Systems 33 (2013), no. 2, 609?623.%
%
%
%
%
%
%
%
%
%

\bibitem{Sturm} Sturm K. T., \emph{Probability measures on metric spaces of nonpositive curvature}, In \emph{Heat kernels and analysis on manifolds, graphs and metric spaces}, Contemp. Math. {\bf 338} (2003), 357--390.




\end{thebibliography}

\end{document}